\documentclass[11pt]{amsart}

\usepackage{mathpazo}
\usepackage{amsmath,amsfonts,amssymb}
\usepackage{amsrefs}
\usepackage{amsthm}
\usepackage{latexsym,amsmath,amssymb,amsfonts}
\usepackage{extarrows}
\usepackage{rotating}
\usepackage{mathrsfs}
\usepackage{xypic} \xyoption{all}
\usepackage{amscd}
\usepackage{hyperref} 
\usepackage{euscript}
\usepackage{hhline}
\usepackage{graphicx,epstopdf}
\usepackage{epsfig}
\usepackage{xcolor}
\usepackage[all,color]{xy}
\usepackage{tikz}
\usetikzlibrary{cd}

\newlength{\fighskip} \fighskip=2pt
\newlength{\figvskip} \figvskip=3pt

\usepackage{hyperref}

\advance\textheight by \topskip
\numberwithin{equation}{section}

\newcommand{\C}{\mathbb{C}}

\newcommand{\R}{\mathbb{R}}

\newcommand{\Z}{\mathbb{Z}}

\newcommand{\EE}{{\mathbb E}}

\newcommand{\bbracket}[1]{\left[#1\right]}

\newcommand{\OO}{{\mathcal O}}
\newcommand{\sO}{\mathcal O^{\#}}

\newcommand{\PP}{\mathrm{P}}

\newcommand{\LR}[1]{\left ( #1 \right )}
\newcommand{\Lr}[1]{\left [ #1 \right ]}

\newcommand{\lR}[1]{\left < #1\right >}
\newcommand{\TopCor}[1]{\left < #1\right >}
\newcommand{\ali}[1]{$$\begin{aligned} #1 \end{aligned}$$}

\newcommand{\XX}{\mathbb X}
\renewcommand{\PP}{\mathbb P}

\newcommand{\lra}{ \longrightarrow}
\newcommand{\lmt}{ \longmapsto}
\newcommand{\w}{\wedge}
\newcommand{\cA}{\mathcal A} 
\newcommand{\xx}{\mathbf{x}}

\newcommand{\xp}{\mathbf{p}}

\newcommand{\pp}[2]{\frac{\partial #1}{\partial #2}} 
\newcommand{\blue}[1]{{\color{blue}{#1}}}

\newcommand{\inv}{^{-1}}
\newcommand{\LL}{\mathscr{L} }
\renewcommand{\O}{{\mathcal O}}

\DeclareMathOperator{\Cyc}{Cyc}

\DeclareMathOperator{\sign}{sign}
\theoremstyle{plain}
\newtheorem{thm}{Theorem}[section]
\newtheorem{thm-defn}{Theorem/Definition}[section]
\newtheorem{lem}[thm]{Lemma}
\newtheorem{lem-defn}[thm]{Lemma/Definition}
\newtheorem{prop}[thm]{Proposition}
\newtheorem{cor}[thm]{Corollary}

\newtheorem*{thm*}{Theorem}
\newtheorem*{prop*}{Proposition}
\newtheorem*{lem*}{Lemma}

\theoremstyle{definition}
\newtheorem{defn}[thm]{Definition}

\newtheorem{eg}[thm]{Example}
\newtheorem*{defn*}{Definition}

\theoremstyle{remark}
\newtheorem{rmk}[thm]{Remark}

\allowdisplaybreaks[4]  

 \title{Stochastic Calculus  and  Hochschild Homology}
  \author{Si Li, Zichang Wang, Peng Yang}

\address{
S.~ Li:
Department of Mathematical Sciences, Tsinghua University, Beijing, China
}
\email{sili@mail.tsinghua.edu.cn}

  \address{
Z.C. ~ Wang: Qiuzhen College, Tsinghua University, Beijing, China
}
\email{zichang-21@mails.tsinghua.edu.cn}

  \address{
P.~ Yang: Department of Mathematical Sciences, Tsinghua University, Beijing, China
}
\email{yang-p20@mails.tsinghua.edu.cn}

\begin{document}

\begin{abstract}
This paper is a case study of probabilistic approach to homological aspects of topological quantum field theory via the example of topological quantum mechanics. We propose topological correlations in terms of large variance limit. An investigation on the relation between probabilistic topological correlations on the circle and  Hochschild homology is illustrated. 
\end{abstract}

\maketitle

\tableofcontents






\section{Introduction}

In physics, mechanics is usually described by an action (which we express on the phase space)
$$
S= \int pdx-H dt.
$$
Here $(x,p)$ are canonical conjugate variables and $H$ is the Hamiltonian. When $H=0$ vanishes, the theory becomes topological. We call this model topological quantum mechanics, which is one simple example  of topological quantum field theories \cite{Atiyah}. Although the Hamiltonian is trivial, correlation functions in this topological model exhibit nontrivial properties. For example, when the time is on the circle $S^1$, topological correlations are related to trace map in deformation quantization and algebraic index \cites{GLL, GLX}. 

The goal of this paper is two-folded. The first is to build up the  probabilistic theory of topological correlation functions in this example as a large variance limit of stochastic process (Definition \ref{defn-Top}). The second is to illustrate the relation between probabilistic topological correlations and aspects of noncommutative geometry. We hope this example would shed some light on deeper connections between stochastic calculus and topological quantum field theories. 

\textbf{Acknowledgment}. The authors would like to thank Chenlin Gu and Elton Hsu for helpful communications. This work of S. L. is supported by the National Key R \& D Program of China (NO. 2020YFA0713000).

\section{Stochastic Calculus on Loops}\label{sec:Stochastic}

\subsection{Probabilistic Preliminaries}
We start by collecting some useful results on random variables that will be frequently used. 

\begin{prop}[\cite{JFLGall}, Proposition 1.1]\label{prop-Gaussian}
If $\{X_m\}_{m\ge 1}$ are Gaussian variables such that
\ali{
X_m\rightarrow X \text{\ \ in } \mathbb{P}
}
then $X$ is also a (possibly degenerate, i.e. variance zero) Gaussian variable, and the convergence holds in  $L^p$ for $1\leq p<\infty$.
\end{prop}
The above proposition  tells that we do not need to distinguish different types of $L^p$ convergence of Gaussian random variables.
We will also need the following  useful proposition proved by Lévy.  
\begin{prop}[\cite{KaiLaiChung}, Theorem 5.3.4]\label{Levy lem}
An infinite sum of independent random variables converges almost surely if and only if it converges in probability.
\end{prop}

\begin{defn}
Let $X=\{X_t\}_{t\in T}$ and $Y=\{Y_t\}_{t\in T}$ be stochastic processes defined on the same index set T and with values in a mutual measure space $(\Omega,\mathcal{F})$. Then $Y$ is a modification of $X$ if for every $t\in T$, $\PP[X_t=Y_t]=1$.
\end{defn}

\begin{prop}[\cite{JFLGall}, Theorem 2.9, Kolmogorov’s Lemma]\label{Kolmogorov’s Lemma}
Let $\{X_t\}_{t \in T}$ be a stochastic process defined on a probability space $(\Omega, \mathcal{F}, \mathbb{P})$, where $T \subset \mathbb{R}$ is a compact interval. Suppose there exist  $\alpha,\beta,C > 0$ such that for all $s, t \in T$,
\begin{equation*}
    \mathbb{E}\left[|X_t - X_s|^\alpha\right] \leq C|t - s|^{1 + \beta},
\end{equation*}
then there exists a modification $\{\tilde{X}_t\}_{t \in T}$ of the process $\{X_t\}$ which is H\"older continuous of order $\gamma$ for every $\gamma \in (0, \beta/\alpha)$: That is, for every $\gamma \in (0, \beta/\alpha)$, there exists a finite constant $M>0$ such that
\begin{equation*}
    |\tilde{X}_t - \tilde{X}_s| \leq M|t - s|^\gamma \quad \text{almost surely}
\end{equation*}
for every $s, t \in T$.
\end{prop}

\subsection{Gaussian Free Field on $S^1$}
We place the time $t$ on the circle $S^1$, which is identified with $\R/\Z$.

\begin{prop}[Gaussian Free Field on $S^1$]\label{propGFF}
Let $X_m$, $Y_m$, $P_m$, $Q_m$ be i.i.d. normal distributions with variance $\sigma^2>0$. Then the following infinite sums 
\ali{
\XX(t)& =\sum_{m> 0} \frac{1}{\sqrt 2\pi m} \LR{\cos\LR{2\pi mt} X_m + \sin\LR{2\pi mt} Y_m} \\
\PP(t)& =\sum_{m> 0} \frac{1}{\sqrt 2\pi m} \LR{\cos\LR{2\pi mt} P_m + \sin\LR{2\pi mt} Q_m}
}
$L^2$-converge for every $t$. Moreover, there exist  continuous modifications of processes $\XX(t)$ and $\PP(t)$.
\end{prop}
\begin{proof} 
Observe that
\ali{
\sum_{m>0} \LR{\frac{1}{\sqrt 2\pi m}}^2 \EE[\LR{\cos\LR{2\pi mt} X_m + \sin\LR{2\pi mt} Y_m}^2]= \sum_{m> 0} \frac{\sigma^2}{2 \pi^2 m^2}=\frac{\sigma^2}{12}<\infty.
}
Being a $L^2$-limit of Gaussian variables for every $t$,  $\XX(t)$ is a Gaussian process. 

Now for any $0\leq t\leq s< 1$, we compute
\ali{
\EE[\XX(t)\XX(s)]&=\sum_{m>0}\frac{\sigma^2}{2\pi^2m^2}(\cos(2\pi mt)\cos(2\pi ms)+\sin(2\pi mt)\sin(2\pi ms))\\
&=\sum_{m>0}\frac{\sigma^2\cos(2\pi m(s-t))}{2\pi^2m^2}\\
&=-\frac{\sigma^2}{2}(s-t)\LR{1-(s-t)}+\frac{\sigma^2}{12}.
} 
Since $\XX(t)-\XX(s)$ is a centered Gaussian, we have
\ali{
\EE[|\XX(t)-\XX(s)|^4]&=3\EE[|\XX(t)-\XX(s)|^2]^2\\
&=3\sigma^4\LR{\frac{1}{6}-\frac{1}{6}+(s-t)\LR{1-(s-t)}}\\
&=3\sigma^4(s-t)^2\LR{1-(s-t)}^2\\
&\leq 3\sigma^4(s-t)^2.
}

Using Proposition \ref{Kolmogorov’s Lemma}, we find a continuous modification of the process $\XX(t)$. Similar things hold for $\PP(t)$.

\end{proof}

Consider  Gaussian free fields $\XX(t),\PP(t)$ on $S^1$ defined in  Proposition \ref{propGFF}. Motivated by the Fourier expansion, we define the  topological action by 
\ali{
\int_{S^1} \PP d\XX
:=-\sum_{m> 0} \frac{1}{2\pi m} (X_mQ_m-Y_mP_m).
}
The sum converges in $L^2$ since
\ali{
\sum_{m> 0} \EE\Lr{\LR{\frac{1}{2\pi m} (X_mQ_m-Y_mP_m)}^2}=\sum_{m> 0} \frac{\sigma^4}{2\pi^2m^2}
<\infty.
}
By Proposition  \ref{Levy lem},  it also converges almost surely.

\begin{prop}\label{2.4} In the limit $N\to \infty$, 
\ali{
\exp\LR{-\frac{i}{\hbar}\sum_{m=1}^{N} \frac{1}{2\pi m} (X_mQ_m-Y_mP_m)}
\rightarrow \exp\LR{\frac{i}{\hbar} \int_{S^1} \PP d\XX}
\text{\ \ in $L^2$}.
}
\end{prop}
\begin{proof}
Let $T_N=-\frac{1}{\hbar}\sum\limits_{m=1}^{N} \frac{1}{2\pi m} (X_mQ_m-Y_mP_m)$, $T=\frac{1}{\hbar}\int_{S^1} \PP d\XX$.  
We have
\ali{
\EE\Lr{\left|\exp\LR{iT_N}-\exp\LR{iT}\right|^2}
&=\EE\Lr{\LR{\sin{T_N}-\sin{T}}^2+\LR{\cos{T_N}-\cos{T}}^2} \\
&\leq\EE\Lr{\LR{T_N-T}^2+\LR{T_N-T}^2}
\rightarrow 0.
}
\end{proof}
\begin{prop}
$$
\EE\Lr{e^{\frac{i}{\hbar} \int_{S^1} \PP d\XX}}=\frac{\frac{\sigma^2}{2\hbar}}{\sinh\LR{\frac{\sigma^2}{2\hbar}}}.
$$
\end{prop}
\begin{proof}
By Proposition \ref{2.4}, we have
\ali{
 \EE\Lr{e ^{\frac{i}{\hbar} \int_{S^1} \PP d\XX}}
=&\lim_{N\rightarrow\infty}\EE\Lr{\exp\LR{-\frac{i}{\hbar}\sum\limits_{m=1}^{N} \frac{1}{2\pi m} (X_mQ_m-Y_mP_m)}} \\
=&\prod_{m>0} \EE\Lr{\exp\LR{-\frac{i}{\hbar} \frac{1}{2\pi m} (X_mQ_m-Y_mP_m)}} \\
=&\prod_{m>0} 
  \int dxdq \, \frac{1}{2\pi \sigma^2}e^{-\frac{x^2+q^2}{2\sigma^2}-\frac{i}{2\pi m\hbar}xq}  
  \int dydp \, \frac{1}{2\pi \sigma^2}e^{-\frac{y^2+p^2}{2\sigma^2}+\frac{i}{2\pi m\hbar}yp}  \\
=&\prod_{m>0} \frac{1}{1+\LR{\frac{\sigma^2}{2\pi m\hbar}}^{2}}
=\frac{\frac{\sigma^2}{2\hbar}}{\sinh\LR{\frac{\sigma^2}{2\hbar}}}.
}
\end{proof}

The above construction extends straightforwardly to the vector-valued fields (with independent components)
$$
\XX=(\XX^1,\cdots,\XX^r),\qquad \PP=(\PP_1,\cdots,\PP_r)
$$ 
and topological action   
$$
\int_{S^1}\sum_{i=1}^r \PP_i d\XX^i.
$$
This can be viewed as topological quantum mechanical model on the phase space $\R^{2r}$. 

\subsection{Topological Correlations via Large Variance Limit}\label{sec:Topological-correlation}
In this section, we define and study topological correlations in our topological quantum mechanical model via large variance limit. 

Let $\R^{2r}$ be the standard phase space with coordinates 
$$
(\xx,\xp)=(x^1,\cdots,x^r,p_1,\cdots,p_r)
$$
and 
$$(\XX,\PP)=(\XX^1,\cdots,\XX^r,\PP_1,\cdots,\PP_r)$$ 
be  vector-valued Gaussian free fields. Given a function $f(\xx,\xp)$ on $\R^{2r}$, we associate the following process on $S^1$ (here $t\in S^1$)
$$
\O_f(t) := f(\XX(t),\PP(t)), \qquad \sO_{f}(t;\xx,\xp) 
:= f(\xx+\XX(t),\xp+\PP(t)).
$$ 

For simplicity, we will often write $\sO_{f}(t)$ for $\sO_{f}(t;\xx,\xp)$ and the dependence on $(\xx,\xp)$ will be implicit in the context.  Thus $\sO_{f}(t) $ can be viewed as the random perturbation of $f$ at time $t$.

\begin{defn}[Topological Correlation]\label{defn-Top} 
Given functions $f_i$ on $\R^{2r}$, we define their topological correlations by the following large variance limit
\ali{
\lR{\OO_{f_0}(t_0)\OO_{f_1}(t_1)\cdots\OO_{f_n}(t_n)}
:=\lim_{\sigma\to\infty} \frac{\EE\Lr{\OO_{f_0}(t_0)\OO_{t_1}(t_1)\cdots\OO_{f_n}(t_n) e^{\frac{i}{\hbar} \int_{S^1} \sum\limits_{i=1}^r \PP_i d\XX^i}}}{\EE\Lr{e^{\frac{i}{\hbar} \int_{S^1} \sum\limits_{i=1}^r \PP_i d\XX^i}}}.
}
and
\ali{
\lR{\sO_{{f_0}}(t_0)\sO_{{f_1}}(t_1)\cdots\sO_{{f_n}}(t_n)}(\xx,\xp)
:=\lim_{\sigma\to\infty} \frac{\EE\Lr{\sO_{{f_0}}(t_0;\xx,\xp)\sO_{{f_1}}(t_1;\xx,\xp)\cdots\sO_{{f_n}}(t_n;\xx,\xp)  e^{\frac{i}{\hbar} \int_{S^1}\sum\limits_{i=1}^r \PP_i d\XX^i}}}{\EE\Lr{e^{\frac{i}{\hbar} \int_{S^1}\sum\limits_{i=1}^r \PP_i d\XX^i}}}.
}
\end{defn}

In the rest of this subsection, we show that the large variance limit of topological correlations exists for polynomial and Schwartz functions, and study their basic properties.

\subsubsection{Polynomials and Wick's Theorem} We first consider topological correlations for polynomial functions. It is enough to consider topological correlations of the form
$$
\lR{\XX^{i_1}(t_1)\XX^{i_2}(t_2)\cdots\XX^{i_u}(t_u)\PP_{j_1}(s_1)\PP_{j_2}(s_2)\cdots\PP_{j_v}(s_v)} .
$$
This can be computed exactly in terms of a version of Wick's Theorem. We assume $r=1$ and consider the phase space $\R^2$ for simplicity. The results generalize to $\R^{2r}$ straight-forwardly and we collect them at the end.

 Let us first define a two-variable function $G$ on $S^1\times S^1$ which is actually the Green's function representing inverse of $\frac{d}{dt}$ on $S^1$. Since every two points on $S^1$ can be rearranged to $t,s\in\R$ such that $0\leq s-t<1$, we only need to define $G$ for points $0\leq s-t<1$ which is given by
\ali{
G(t,s)
&=-\sum_{n> 0} \frac{1}{\pi n} \sin\LR{2\pi n(s-t)}
=\begin{cases}
s-t-\frac{1}{2} & 0<s-t<1 \\
0 & s-t=0
\end{cases}
}

\begin{figure}[h]
\begin{center}
\tikzset{every picture/.style={line width=0.75pt}} 

\begin{tikzpicture}[x=0.75pt,y=0.75pt,yscale=-1,xscale=1,scale=0.9]

\draw    (39.83,164) -- (558.17,164) ;
\draw [shift={(560.17,164)}, rotate = 180] [color={rgb, 255:red, 0; green, 0; blue, 0 }  ][line width=0.75]    (10.93,-3.29) .. controls (6.95,-1.4) and (3.31,-0.3) .. (0,0) .. controls (3.31,0.3) and (6.95,1.4) .. (10.93,3.29)   ;
\draw    (300,274) -- (300,56) ;
\draw [shift={(300,54)}, rotate = 90] [color={rgb, 255:red, 0; green, 0; blue, 0 }  ][line width=0.75]    (10.93,-3.29) .. controls (6.95,-1.4) and (3.31,-0.3) .. (0,0) .. controls (3.31,0.3) and (6.95,1.4) .. (10.93,3.29)   ;
\draw [color={rgb, 255:red, 0; green, 0; blue, 255 }  ,draw opacity=1 ]   (300,223) -- (418,105) ;
\draw  [color={rgb, 255:red, 0; green, 0; blue, 255 }  ,draw opacity=1 ][fill={rgb, 255:red, 255; green, 255; blue, 255 }  ,fill opacity=1 ] (295.5,223) .. controls (295.5,220.51) and (297.51,218.5) .. (300,218.5) .. controls (302.49,218.5) and (304.5,220.51) .. (304.5,223) .. controls (304.5,225.49) and (302.49,227.5) .. (300,227.5) .. controls (297.51,227.5) and (295.5,225.49) .. (295.5,223) -- cycle ;
\draw  [color={rgb, 255:red, 0; green, 0; blue, 255 }  ,draw opacity=1 ][fill={rgb, 255:red, 0; green, 0; blue, 255 }  ,fill opacity=1 ] (413.5,164) .. controls (413.5,161.51) and (415.51,159.5) .. (418,159.5) .. controls (420.49,159.5) and (422.5,161.51) .. (422.5,164) .. controls (422.5,166.49) and (420.49,168.5) .. (418,168.5) .. controls (415.51,168.5) and (413.5,166.49) .. (413.5,164) -- cycle ;
\draw [color={rgb, 255:red, 0; green, 0; blue, 255 }  ,draw opacity=1 ]   (418,223) -- (536,105) ;
\draw  [color={rgb, 255:red, 0; green, 0; blue, 255 }  ,draw opacity=1 ][fill={rgb, 255:red, 255; green, 255; blue, 255 }  ,fill opacity=1 ] (413.5,105) .. controls (413.5,102.51) and (415.51,100.5) .. (418,100.5) .. controls (420.49,100.5) and (422.5,102.51) .. (422.5,105) .. controls (422.5,107.49) and (420.49,109.5) .. (418,109.5) .. controls (415.51,109.5) and (413.5,107.49) .. (413.5,105) -- cycle ;
\draw  [color={rgb, 255:red, 0; green, 0; blue, 255 }  ,draw opacity=1 ][fill={rgb, 255:red, 255; green, 255; blue, 255 }  ,fill opacity=1 ] (531.5,105) .. controls (531.5,102.51) and (533.51,100.5) .. (536,100.5) .. controls (538.49,100.5) and (540.5,102.51) .. (540.5,105) .. controls (540.5,107.49) and (538.49,109.5) .. (536,109.5) .. controls (533.51,109.5) and (531.5,107.49) .. (531.5,105) -- cycle ;
\draw  [color={rgb, 255:red, 0; green, 0; blue, 255 }  ,draw opacity=1 ][fill={rgb, 255:red, 255; green, 255; blue, 255 }  ,fill opacity=1 ] (413.5,223) .. controls (413.5,220.51) and (415.51,218.5) .. (418,218.5) .. controls (420.49,218.5) and (422.5,220.51) .. (422.5,223) .. controls (422.5,225.49) and (420.49,227.5) .. (418,227.5) .. controls (415.51,227.5) and (413.5,225.49) .. (413.5,223) -- cycle ;
\draw  [color={rgb, 255:red, 0; green, 0; blue, 255 }  ,draw opacity=1 ][fill={rgb, 255:red, 0; green, 0; blue, 255 }  ,fill opacity=1 ] (531.5,164) .. controls (531.5,161.51) and (533.51,159.5) .. (536,159.5) .. controls (538.49,159.5) and (540.5,161.51) .. (540.5,164) .. controls (540.5,166.49) and (538.49,168.5) .. (536,168.5) .. controls (533.51,168.5) and (531.5,166.49) .. (531.5,164) -- cycle ;
\draw  [color={rgb, 255:red, 0; green, 0; blue, 255 }  ,draw opacity=1 ][fill={rgb, 255:red, 0; green, 0; blue, 255 }  ,fill opacity=1 ] (295.5,164) .. controls (295.5,161.51) and (297.51,159.5) .. (300,159.5) .. controls (302.49,159.5) and (304.5,161.51) .. (304.5,164) .. controls (304.5,166.49) and (302.49,168.5) .. (300,168.5) .. controls (297.51,168.5) and (295.5,166.49) .. (295.5,164) -- cycle ;
\draw [color={rgb, 255:red, 0; green, 0; blue, 255 }  ,draw opacity=1 ]   (182,223) -- (300,105) ;
\draw  [color={rgb, 255:red, 0; green, 0; blue, 255 }  ,draw opacity=1 ][fill={rgb, 255:red, 0; green, 0; blue, 255 }  ,fill opacity=1 ] (177.5,164) .. controls (177.5,161.51) and (179.51,159.5) .. (182,159.5) .. controls (184.49,159.5) and (186.5,161.51) .. (186.5,164) .. controls (186.5,166.49) and (184.49,168.5) .. (182,168.5) .. controls (179.51,168.5) and (177.5,166.49) .. (177.5,164) -- cycle ;
\draw  [color={rgb, 255:red, 0; green, 0; blue, 255 }  ,draw opacity=1 ][fill={rgb, 255:red, 255; green, 255; blue, 255 }  ,fill opacity=1 ] (295.5,105) .. controls (295.5,102.51) and (297.51,100.5) .. (300,100.5) .. controls (302.49,100.5) and (304.5,102.51) .. (304.5,105) .. controls (304.5,107.49) and (302.49,109.5) .. (300,109.5) .. controls (297.51,109.5) and (295.5,107.49) .. (295.5,105) -- cycle ;
\draw [color={rgb, 255:red, 0; green, 0; blue, 255 }  ,draw opacity=1 ]   (64,223) -- (182,105) ;
\draw  [color={rgb, 255:red, 0; green, 0; blue, 255 }  ,draw opacity=1 ][fill={rgb, 255:red, 255; green, 255; blue, 255 }  ,fill opacity=1 ] (177.5,105) .. controls (177.5,102.51) and (179.51,100.5) .. (182,100.5) .. controls (184.49,100.5) and (186.5,102.51) .. (186.5,105) .. controls (186.5,107.49) and (184.49,109.5) .. (182,109.5) .. controls (179.51,109.5) and (177.5,107.49) .. (177.5,105) -- cycle ;
\draw  [color={rgb, 255:red, 0; green, 0; blue, 255 }  ,draw opacity=1 ][fill={rgb, 255:red, 255; green, 255; blue, 255 }  ,fill opacity=1 ] (177.5,223) .. controls (177.5,220.51) and (179.51,218.5) .. (182,218.5) .. controls (184.49,218.5) and (186.5,220.51) .. (186.5,223) .. controls (186.5,225.49) and (184.49,227.5) .. (182,227.5) .. controls (179.51,227.5) and (177.5,225.49) .. (177.5,223) -- cycle ;
\draw  [color={rgb, 255:red, 0; green, 0; blue, 255 }  ,draw opacity=1 ][fill={rgb, 255:red, 0; green, 0; blue, 255 }  ,fill opacity=1 ] (59.5,164) .. controls (59.5,161.51) and (61.51,159.5) .. (64,159.5) .. controls (66.49,159.5) and (68.5,161.51) .. (68.5,164) .. controls (68.5,166.49) and (66.49,168.5) .. (64,168.5) .. controls (61.51,168.5) and (59.5,166.49) .. (59.5,164) -- cycle ;
\draw  [color={rgb, 255:red, 0; green, 0; blue, 255 }  ,draw opacity=1 ][fill={rgb, 255:red, 255; green, 255; blue, 255 }  ,fill opacity=1 ] (59.5,223) .. controls (59.5,220.51) and (61.51,218.5) .. (64,218.5) .. controls (66.49,218.5) and (68.5,220.51) .. (68.5,223) .. controls (68.5,225.49) and (66.49,227.5) .. (64,227.5) .. controls (61.51,227.5) and (59.5,225.49) .. (59.5,223) -- cycle ;
\draw [color={rgb, 255:red, 0; green, 0; blue, 255 }  ,draw opacity=1 ]   (536,223) -- (560.17,198.83) ;
\draw [color={rgb, 255:red, 0; green, 0; blue, 255 }  ,draw opacity=1 ]   (39.83,129.17) -- (64,105) ;
\draw  [color={rgb, 255:red, 0; green, 0; blue, 255 }  ,draw opacity=1 ][fill={rgb, 255:red, 255; green, 255; blue, 255 }  ,fill opacity=1 ] (59.5,105) .. controls (59.5,102.51) and (61.51,100.5) .. (64,100.5) .. controls (66.49,100.5) and (68.5,102.51) .. (68.5,105) .. controls (68.5,107.49) and (66.49,109.5) .. (64,109.5) .. controls (61.51,109.5) and (59.5,107.49) .. (59.5,105) -- cycle ;
\draw  [color={rgb, 255:red, 0; green, 0; blue, 255 }  ,draw opacity=1 ][fill={rgb, 255:red, 255; green, 255; blue, 255 }  ,fill opacity=1 ] (531.5,223) .. controls (531.5,220.51) and (533.51,218.5) .. (536,218.5) .. controls (538.49,218.5) and (540.5,220.51) .. (540.5,223) .. controls (540.5,225.49) and (538.49,227.5) .. (536,227.5) .. controls (533.51,227.5) and (531.5,225.49) .. (531.5,223) -- cycle ;

\draw (282,38) node [anchor=north west][inner sep=0.75pt]   [align=left] {$G(t,s)$};
\draw (560,164) node [anchor=north west][inner sep=0.75pt]   [align=left] {$s-t$};
\draw (302,167) node [anchor=north west][inner sep=0.75pt]   [align=left] {$0$};
\draw (355.5,167) node [anchor=north west][inner sep=0.75pt]   [align=left] {$\frac{1}{2}$};
\draw (420,167) node [anchor=north west][inner sep=0.75pt]   [align=left] {$1$};
\draw (240,167) node [anchor=north west][inner sep=0.75pt]   [align=left] {$-\frac{1}{2}$};
\draw (184,167) node [anchor=north west][inner sep=0.75pt]   [align=left] {$-1$};

\end{tikzpicture}
\end{center}
\caption{The value of $G(t,s)$ as a periodic function of $s-t$}
\end{figure}
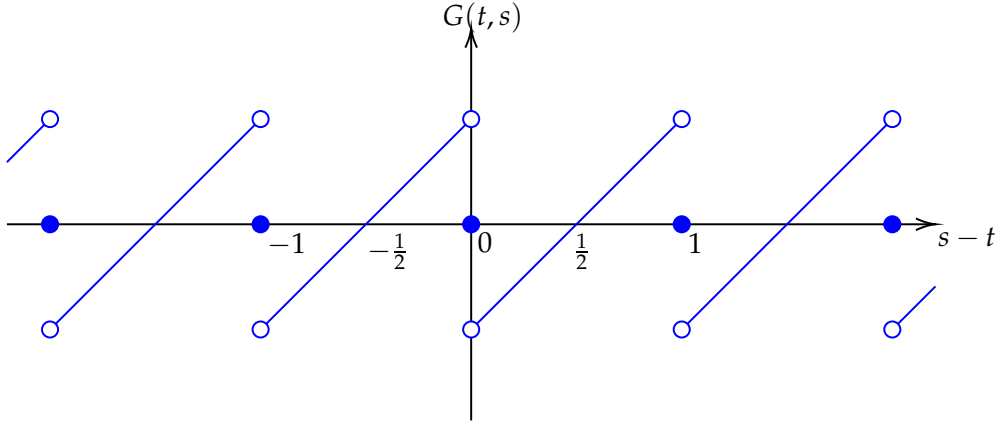

\begin{lem}\label{lem-generating}
For any $J=(J_1, J_2,\cdots, J_u)^T, K=(K_1, K_2,\cdots, K_v)^T$ and $0\leq t_1, t_2,\cdots, t_u, s_1, s_2,\cdots, s_v< 1$, denote $Q=(Q_{kj})$ to be the $u\times v$ matrix with $Q_{kj}=i\hbar  G(t_k,s_j)$. Then we have the topological correlation
\ali{
\lR{e^{\sum\limits_{k=1}^u\XX(t_k)J_k+\sum\limits_{j=1}^v\PP(s_j)K_j}}=e^{J^TQK}.
}
\end{lem}
\begin{proof}
By definition
\ali{
\lR{e^{\sum\limits_{k=1}^u\XX(t_k)J_k+\sum\limits_{j=1}^v\PP(s_j)K_j}}=\lim_{\sigma\rightarrow\infty}
\frac{\EE\Lr{e^{\sum\limits_{k=1}^u\XX(t_k)J_k+\sum\limits_{j=1}^v\PP(s_j)K_j+\frac{i}{\hbar}\int_{S^1} \PP d\XX}}}
{\EE\Lr{e^{\frac{i}{\hbar}\int_{S^1} \PP d\XX}}}.
}
For any $k,l\geq0$, $t,t_1,\cdots t_k,s_1,\cdots s_l\in[0,1]$, since
\ali{
T_N:=&\exp\LR{-\frac{i}{\hbar}\sum_{m=1}^{N} \frac{1}{2\pi m} (X_mQ_m-Y_mP_m)}
\rightarrow \exp\LR{\frac{i}{\hbar} \int_{S^1}\PP d\XX}
\text{\ \ in $L^2$}
\\
V_N(t):=&\sum_{m=1}^{N} \frac{1}{\sqrt 2\pi m} \LR{\cos\LR{2\pi mt} X_m + \sin\LR{2\pi mt} Y_m}
\rightarrow \XX(t)
\text{\ \ in any $L^p$}
\\
W_N(t):=&\sum_{m=1}^{N} \frac{1}{\sqrt 2\pi m} \LR{\cos\LR{2\pi mt} P_m + \sin\LR{2\pi mt} Q_m}
\rightarrow \PP(t)
\text{\ \ in any $L^p$}
}
we have
\ali{
\prod_{a=1}^{k}V_N(t_a)
\prod_{b=1}^{l}W_N(t_b)
T_N
\rightarrow \prod_{a=1}^{k}\XX(t_a)\prod_{b=1}^{l}\PP(t_b)\exp\LR{\frac{i}{\hbar} \int_{S^1}\PP d\XX}
\text{\ \ in $L^1$}.
}
In particular, the expectation converges. Now we can calculate
\ali{
&\EE\Lr{e^{\sum\limits_{k=1}^u\XX(t_k)J_k+\sum\limits_{j=1}^v\PP(s_j)K_j+\frac{i}{\hbar}\int_{S^1} \PP d\XX}}\\
=&\EE\Lr{\prod_{m>0}e^{\LR{\sum\limits_{k=1}^u\frac{\cos(2\pi m t_k)}{\sqrt2 \pi m}J_k}X_m+\LR{\sum\limits_{k=1}^u\frac{\sin(2\pi m t_k)}{\sqrt2 \pi m}J_k}Y_m+\LR{\sum\limits_{j=1}^v\frac{\sin(2\pi m s_j)}{\sqrt2 \pi m}K_j}Q_m+\LR{\sum\limits_{j=1}^v\frac{\cos(2\pi m s_j)}{\sqrt2 \pi m}K_j}P_m-\frac{i}{2\pi m\hbar}(X_mQ_m-Y_mP_m)}}\\
=&\prod_{m>0}\EE\Lr{e^{ \LR{\sum\limits_{k=1}^u\frac{\sin(2\pi m t_k)}{\sqrt2 \pi m}J_k}Y_m +\LR{\sum\limits_{j=1}^v\frac{\cos(2\pi m s_j)}{\sqrt2 \pi m}K_j}P_m+\frac{i}{2\pi m\hbar}Y_mP_m}}\EE\Lr{e^{\LR{\sum\limits_{k=1}^u\frac{\cos(2\pi m t_k)}{\sqrt2 \pi m}J_k}X_m+\LR{\sum\limits_{j=1}^v\frac{\cos(2\pi m s_j)}{\sqrt2 \pi m}K_j}P_m-\frac{i}{2\pi m\hbar}X_mQ_m}}.
}

Let us compute $\EE\Lr{e^{ \LR{\sum\limits_{k=1}^u\frac{\sin(2\pi m t_k)}{\sqrt2 \pi m}J_k}Y_m +\LR{\sum\limits_{j=1}^v\frac{\cos(2\pi m s_j)}{\sqrt2 \pi m}K_j}P_m+\frac{i}{2\pi m\hbar}Y_mP_m}}$ first. Denote 
$$
a_m=\sum_{k=1}^u\frac{\sin(2\pi m t_k)}{\sqrt2 \pi m}J_k,\ 
c_m=\sum_{j=1}^v\frac{\cos(2\pi m s_j)}{\sqrt2 \pi m}K_j 
$$
and
$$A_m=-\frac{1}{\frac{1}{\sigma^4}+\frac{1}{4\pi^2m^2\hbar^2}}\LR{\frac{1}{\sigma^2}a_m+\frac{i}{2\pi m\hbar}c_m},\ C_m=-\frac{1}{\frac{1}{\sigma^4}+\frac{1}{4\pi^2m^2\hbar^2}}\LR{\frac{1}{\sigma^2}c_m+\frac{i}{2\pi m\hbar}a_m} 
$$
which satisfy
$$
\begin{cases}
        -\frac{1}{\sigma^2}A_m+\frac{i}{2\pi m\hbar}C_m=a_m \\
        \frac{i}{2\pi m\hbar}A_m-\frac{1}{\sigma^2}C_m=c_m \\ 
\end{cases}.
$$
Then  
\ali{
\EE\Lr{e^{a_mY_m+c_mP_m+\frac{i}{2\pi m\hbar}Y_mP_m}}
=&\int_{\R^2}dydp\,\frac{1}{2\pi\sigma^2}e^{-\frac{y^2+p^2}{2\sigma^2}+\frac{i}{2\pi m\hbar}yp+a_my+c_mp}\\
=&\int_{\R^2}dydp\,\frac{1}{2\pi\sigma^2}e^{-\frac{\LR{y+A_m}^2+\LR{p+C_m}^2}{2\sigma^2}+\frac{i}{2\pi m\hbar}(y+A_m)(p+C_m)+\frac{A_m^2+C_m^2}{2\sigma^2}-\frac{i}{2\pi m\hbar}A_mC_m}\\
=&e^{\frac{A_m^2+C_m^2}{2\sigma^2}-\frac{i}{2\pi m\hbar}A_mC_m}\int_{\R^2}dydp\,\frac{1}{2\pi\sigma^2}e^{-\frac{y^2+p^2}{2\sigma^2}+\frac{i}{2\pi m\hbar}yp}.
}
When $\sigma\rightarrow\infty$,
$$
\sum_{m>0}\frac{A_m^2+C_m^2}{2\sigma^2}-\frac{i}{2\pi m\hbar}A_mC_m
=-\sum_{m>0}\LR{4\pi^2m^2\hbar^2}^2\LR{\frac{i}{2\pi m\hbar}}^3a_mc_m
+O(\sigma^{-1}),
$$
so
\ali{
\prod_{m>0}e^{\frac{A_m^2+C_m^2}{2\sigma^2}-\frac{i}{2\pi m\hbar}A_mC_m}
&\rightarrow \exp\LR{-\sum_{m>0}\LR{4\pi^2m^2\hbar^2}^2\LR{\frac{i}{2\pi m\hbar}}^3a_mc_m}\\
&=\exp\LR{\sum_{m>0}\sum_{k=1}^u\sum_{j=1}^v\frac{i\hbar}{\pi m}\cos(2\pi mt_k)\sin(2\pi ms_j)J_kK_j}.
}

We can perform the same calculation for $X_m$ and $Q_m$. Since 
$$
\EE\Lr{e^{\frac{i}{\hbar}\int_{S^1} \PP d\XX}}=\prod_{m>0}\int_{\R^2}dxdq\,\frac{1}{2\pi\sigma^2}e^{-\frac{x^2+q^2}{2\sigma^2}-\frac{i}{2\pi m\hbar}xq}\int_{\R^2}dydp\,\frac{1}{2\pi\sigma^2}e^{-\frac{y^2+p^2}{2\sigma^2}+\frac{i}{2\pi m\hbar}yp},
$$
we find
\ali{
&\lR{e^{\sum\limits_{k=1}^u\XX(t_k)J_k+\sum\limits_{j=1}^v\PP(s_j)K_j}}\\
=&\exp\LR{\sum_{m>0}\sum_{k=1}^u\sum_{j=1}^v\frac{i\hbar}{\pi m}\LR{\cos(2\pi mt_k)\sin(2\pi ms_j)-\sin(2\pi mt_k)\cos(2\pi ms_j)}J_kK_j}\\
=&\exp\LR{\sum_{m>0}\sum_{k=1}^u\sum_{j=1}^v\frac{i\hbar}{\pi m}\sin(2\pi m(s_j-t_k))J_kK_j}\\
=&\exp\LR{\sum_{k=1}^u\sum_{j=1}^v Q_{kj}J_kK_j}.
}
\end{proof}

\begin{prop}[Wick's Theorem-$\R^2$]\label{prop-Wick}
For any $u+v$ points 
$t_1,t_2,\cdots,t_u,s_1,s_2,\cdots,s_v\in S^1$,
\ali{
&\quad \lR{\XX(t_1)\XX(t_2)\cdots\XX(t_u)\PP(s_1)\PP(s_2)\cdots\PP(s_v)} =\delta_{u,v}\LR{i\hbar}^u \sum_{\lambda\in S_u}\prod_{k=1}^u G(t_k,s_{\lambda(k)}).
}
\end{prop}

\begin{proof}
Expand the exponential in Lemma \ref{lem-generating}.  
Wick's Theorem follows by comparing the coefficients of the monomials consisting of $J_k$'s and $K_j$'s.
\end{proof}

\begin{cor}[Propagator]
The two-point correlations are 
\ali{
\lR{\XX(t_1)\PP(t_2)}
&=i\hbar G(t_1,t_2)   \\
\lR{\XX(t_1)\XX(t_2)}
&=0 \\
\lR{\PP(t_1)\PP(t_2)}
&=0 
}
These two-point correlations are called propagators, which are the building blocks in Wick's Theorem. 
\end{cor}

The above results extend straightforwardly to the phase space $\R^{2r}$ and the vector-valued fields (with independent components)
$$
\XX=(\XX^1,\cdots,\XX^r),\qquad \PP=(\PP_1,\cdots,\PP_r).
$$ 
The non-trivial two-point topological correlations are
$$
\lR{\XX^k(t_1) \PP_j(t_2)}
= i\hbar G(t_1,t_2) \delta^k_j.
$$

\begin{prop}[Wick's Theorem-$\R^{2r}$]\label{prop-Wick-2}
For any $u+v$ points 
$t_1,t_2,\cdots,t_u,s_1,s_2,\cdots,s_v\in S^1$,
\ali{
&\quad \lR{\XX^{i_1}(t_1)\XX^{i_2}(t_2)\cdots\XX^{i_u}(t_u)\PP_{j_1}(s_1)\PP_{j_2}(s_2)\cdots\PP_{j_v}(s_v)} 
=\delta_{u,v}\LR{i\hbar}^u \sum_{\lambda\in S_u}\prod_{k=1}^u G(t_k,s_{\lambda(k)})\delta^{i_k}_{j_{\lambda(k)}}.
}
\end{prop}

The following limit property of two-point topological correlations will be used in Section \ref{sec:q-HKR}

\begin{lem}\label{lem-limit}
In the limit $t_1\to t_2$
\ali{
\lim_{t_1\to t_2^-} 
\lR{\XX^k(t_1) \PP_j(t_2)}
&= -\frac{i\hbar}{2} \delta^k_j \\
\lim_{t_1\to t_2^+} 
\lR{\XX^k(t_1) \PP_j(t_2)}
&= \phantom{-}\frac{i\hbar}{2} \delta^k_j.
}
\end{lem}
\begin{proof} This follows from the corresponding property of $G(t_1,t_2)$. 

\end{proof}



\subsubsection{Schwartz functions and Moyal product}\label{section:Moyal-product}
We next consider topological correlation of random perturbations of Schwartz functions on $\R^{2r}$. We explain how Moyal product arises in this setting. 

Let $f_i(\xx,\xp)$ be Schwartz functions on $\R^{2r}$ with Fourier transforms 
$$
f_i(\xx,\xp)=\int_{\R^{2r}} d\xx'd\xp' \, \hat f(\xx',\xp')e^{2\pi i(\xx\xx'+\xp\xp')}
,\quad 
\hat f_i(\xx',\xp')=\int_{\R^{2r}} d\xx d\xp \,   f(\xx,\xp)e^{-2\pi i(\xx\xx'+\xp\xp')}.
$$
Then we can compute
\ali{
&\phantom{=}\lR{\prod_{j=0}^n \sO_{f_j}(t_j)} 
=\lR{\prod_{j=0}^n f_j(\xx+\XX(t_j),\xp+\PP(t_j))}\\
&=\lR{\prod_{j=0}^n \int_{\R^{2r}}d\xx_j'd\xp_j'\, \hat f_j(\xx_j',\xp_j')\exp\LR{2\pi i \LR{\xx_j'(\xx+\XX(t_j))+\xp_j'(\xp+\PP(t_j))}}}\\
&=\lim_{\sigma\to\infty}\frac{1}{\EE\Lr{e^{\frac{i}{\hbar}\int_{S^1}\sum\limits_{i=1}^r \PP_i d\XX^i }}}
\int_{\R^{2}}\prod_{j=0}^n d\xx_j'd\xp_j'\, \hat f_j(\xx_j',\xp_j') e^{2\pi i (\xx\xx_j'+\xp\xp_j')} \cdot \\
&\qquad\qquad\qquad  \EE\Lr{e^{\frac{i}{\hbar}\int_{S^1}\sum\limits_{i=1}^r \PP_i d\XX^i} \prod_{j=0}^n \exp\LR{2\pi i \LR{\xx_j'\XX(t_i)+\xp_j'\PP(t_i)}}} \\
&\xlongequal{\text{Lemma \ref{lem-generating}}} \lim_{\sigma\to\infty}
\int_{\R^{2r}}\prod_{j=0}^n d\xx_j'd\xp_j'\, \hat f_j(\xx_j',\xp_j') e^{2\pi i (\xx\xx_j'+\xp\xp_j')}     \cdot \exp\LR{(2\pi i)^2\sum_{j=0}^n\sum_{k=0}^n i\hbar G(t_j,t_k)\xx_j'\xp_k'+O(\sigma^{-1})}\\
&=\int_{\R^{2r}}\prod_{j=0}^n d\xx_j'd\xp_j'\, \hat f_j(\xx_j',\xp_j') e^{2\pi i (\xx\xx_j'+\xp\xp_j')}  \cdot \exp\LR{(2\pi i)^2\sum_{j=0}^n\sum_{k=0}^n i\hbar G(t_j,t_k)\xx_j'\xp_k'}. 
}
Thus for any $t_i\in S^1$, the topological correlation $\lR{\prod\limits_{j=0}^n \sO_{f_j}(t_j)}$ is a Schwartz function on $\R^{2r}$.

Notice that the Green's function on $S^1$ closely resembles that on $\R$ at small scales
$$
\lim_{(t-s)\to 0} \LR{G(t,s)-\frac{1}{2}\sign(t-s)}=0, \qquad   \sign(t):=\begin{cases}
1 & t>0 \\ 0 & t=0 \\-1 & t<0
\end{cases}.
$$
As a result, 
\ali{
\lim_{(t_0-t_1)\to 0^+}\lR{\sO_{{f_0}}(t_0)\sO_{{f_1}}(t_1)} (\xx,\xp)
=\int_{\R^{2r}}\prod_{j=0}^1 d\xx_j'd\xp_j'\, \hat f_j(\xx_j',\xp_j') e^{2\pi i (\xx\xx_j'+\xp\xp_j')}\cdot \exp\LR{ \frac{i\hbar}{2} (2\pi i)^2  (\xx_0'\xp_1'-\xp_0'\xx_1')} 
}
which gives the Moyal product of $f_0$ and $f_1$.

\section{Hochschild Homology}\label{sec:q-HKR}

As an application of the topological correlation function (Definition \ref{defn-Top}), we explain how the geometry of $S^1$ is related to the notion of Hochschild homology and present an explicit construction of a quantum version of the Hochschild-Kostant-Rosenberg (HKR) map.

\subsection{Hochschild Homology}
We collect some basic notions on Hochschild Homology of associative algebras. See for example \cite{Weibel} for further details. 

Let $A$ be an associative algebra over a base field $k$. Define the space of $m$-th Hochschild chains
$$
C_m(A) = A^{\otimes m+1}, \qquad  \otimes = \otimes_k 
$$
which can be viewed as the space of $m+1$ ordered observables valued in $A$ on $S^1$. Define a map 
$$
b:C_m(A)\to C_{m-1}(A)
$$ 
by
\ali{
& b(a_0\otimes a_1\otimes\cdots\otimes a_m) \\ 
=\; &a_0 a_1\otimes a_2\otimes\cdots\otimes a_m 
- a_0\otimes a_1a_2\otimes\cdots\otimes a_m  
+\cdots+(-1)^{m-1}a_0\otimes a_1\otimes\cdots\otimes a_{m-1}a_m \\
& +(-1)^m a_ma_0\otimes a_1\otimes\cdots\otimes a_{m-1}.
}
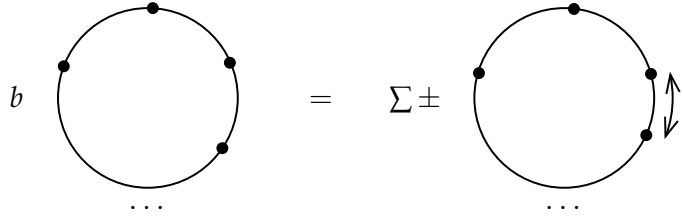
\begin{center}
\tikzset{every picture/.style={line width=0.75pt}} 
\begin{tikzpicture}[x=0.75pt,y=0.75pt,yscale=-1,xscale=1]
\draw   (150,155.33) .. controls (150,130.3) and (170.3,110) .. (195.33,110) .. controls (220.37,110) and (240.67,130.3) .. (240.67,155.33) .. controls (240.67,180.37) and (220.37,200.67) .. (195.33,200.67) .. controls (170.3,200.67) and (150,180.37) .. (150,155.33) -- cycle ;
\draw  [fill={rgb, 255:red, 0; green, 0; blue, 0 }  ,fill opacity=1 ] (230.4,180.69) .. controls (230.4,179.31) and (231.52,178.19) .. (232.9,178.19) .. controls (234.28,178.19) and (235.4,179.31) .. (235.4,180.69) .. controls (235.4,182.07) and (234.28,183.19) .. (232.9,183.19) .. controls (231.52,183.19) and (230.4,182.07) .. (230.4,180.69) -- cycle ;
\draw   (360,155.33) .. controls (360,130.3) and (380.3,110) .. (405.33,110) .. controls (430.37,110) and (450.67,130.3) .. (450.67,155.33) .. controls (450.67,180.37) and (430.37,200.67) .. (405.33,200.67) .. controls (380.3,200.67) and (360,180.37) .. (360,155.33) -- cycle ;
\draw  [fill={rgb, 255:red, 0; green, 0; blue, 0 }  ,fill opacity=1 ] (234.18,137.58) .. controls (234.18,136.2) and (235.3,135.08) .. (236.68,135.08) .. controls (238.06,135.08) and (239.18,136.2) .. (239.18,137.58) .. controls (239.18,138.96) and (238.06,140.08) .. (236.68,140.08) .. controls (235.3,140.08) and (234.18,138.96) .. (234.18,137.58) -- cycle ;
\draw  [fill={rgb, 255:red, 0; green, 0; blue, 0 }  ,fill opacity=1 ] (195.33,110) .. controls (195.33,108.62) and (196.45,107.5) .. (197.83,107.5) .. controls (199.21,107.5) and (200.33,108.62) .. (200.33,110) .. controls (200.33,111.38) and (199.21,112.5) .. (197.83,112.5) .. controls (196.45,112.5) and (195.33,111.38) .. (195.33,110) -- cycle ;
\draw  [fill={rgb, 255:red, 0; green, 0; blue, 0 }  ,fill opacity=1 ] (150.4,139.36) .. controls (150.4,137.97) and (151.52,136.86) .. (152.9,136.86) .. controls (154.28,136.86) and (155.4,137.97) .. (155.4,139.36) .. controls (155.4,140.74) and (154.28,141.86) .. (152.9,141.86) .. controls (151.52,141.86) and (150.4,140.74) .. (150.4,139.36) -- cycle ;
\draw  [fill={rgb, 255:red, 0; green, 0; blue, 0 }  ,fill opacity=1 ] (444.2,174.12) .. controls (444.2,172.74) and (445.32,171.62) .. (446.7,171.62) .. controls (448.08,171.62) and (449.2,172.74) .. (449.2,174.12) .. controls (449.2,175.5) and (448.08,176.62) .. (446.7,176.62) .. controls (445.32,176.62) and (444.2,175.5) .. (444.2,174.12) -- cycle ;
\draw  [fill={rgb, 255:red, 0; green, 0; blue, 0 }  ,fill opacity=1 ] (446.57,143.23) .. controls (446.57,141.85) and (447.69,140.73) .. (449.07,140.73) .. controls (450.45,140.73) and (451.57,141.85) .. (451.57,143.23) .. controls (451.57,144.61) and (450.45,145.73) .. (449.07,145.73) .. controls (447.69,145.73) and (446.57,144.61) .. (446.57,143.23) -- cycle ;
\draw  [fill={rgb, 255:red, 0; green, 0; blue, 0 }  ,fill opacity=1 ] (407.73,110.43) .. controls (407.73,109.05) and (408.85,107.93) .. (410.23,107.93) .. controls (411.61,107.93) and (412.73,109.05) .. (412.73,110.43) .. controls (412.73,111.81) and (411.61,112.93) .. (410.23,112.93) .. controls (408.85,112.93) and (407.73,111.81) .. (407.73,110.43) -- cycle ;
\draw  [fill={rgb, 255:red, 0; green, 0; blue, 0 }  ,fill opacity=1 ] (359.8,142.79) .. controls (359.8,141.41) and (360.92,140.29) .. (362.3,140.29) .. controls (363.68,140.29) and (364.8,141.41) .. (364.8,142.79) .. controls (364.8,144.17) and (363.68,145.29) .. (362.3,145.29) .. controls (360.92,145.29) and (359.8,144.17) .. (359.8,142.79) -- cycle ;
\draw   (458.88,144.28) .. controls (459.62,147.85) and (460,151.55) .. (460,155.33) .. controls (460,161.92) and (458.84,168.23) .. (456.7,174.07) ; 
\draw   (456.83,153.31) -- (458.84,143.54) -- (463.57,152.32) ;
\draw   (462.63,166.96) -- (456.39,174.74) -- (456.17,164.77) ;
\draw (124,148) node [anchor=north west][inner sep=0.75pt]   [align=left] {$b$};
\draw (275,148) node [anchor=north west][inner sep=0.75pt]   [align=left] {$= \quad\,\,\,\,\sum \pm$};
\draw (184,206) node [anchor=north west][inner sep=0.75pt]   [align=left] {$\cdots$};
\draw (394,206) node [anchor=north west][inner sep=0.75pt]   [align=left] {$\cdots$};
\end{tikzpicture}
\end{center}

By  associativity of $A$, we have
\ali{
&b^2(a_0\otimes a_1\otimes\cdots\otimes a_m) \\
=& b\LR{a_0 a_1\otimes a_2\otimes\cdots\otimes a_m 
- a_0\otimes a_1a_2\otimes\cdots\otimes a_m \pm\cdots} \\
=& (a_0 a_1)a_2 \otimes\cdots\otimes a_m - a_0 (a_1 a_2)\otimes\cdots\otimes a_m \pm\cdots \\
=&0.
}
Thus $(C_m(A),b)$ forms a chain complex 
$$
\cdots \overset{b}{\to} C_m(A) \overset{b}{\to} C_{m-1}(A) \overset{b}{\to} \cdots \overset{b}{\to} C_1(A) \overset{b}{\to} C_0(A)
$$
which is called the Hochschild chain complex. Its homology
$$
HH_\bullet(A) := H_\bullet (C_\bullet(A),b)
$$
is called the Hochschild homology. The following theorem is a fundamental result. 

\begin{thm}[Hochschild-Kostant-Rosenberg (HKR)]\label{thm-HKR}
Let $A=k[y^i]$ be a polynomial ring and $\Omega_A^\bullet=k[y^i,dy^i]$ be the algebraic differential forms. Then 
$$
HH_m(A) = \Omega_A^m.
$$
\end{thm}

HKR theorem says that Hochschild homology plays the role of differential forms on ordinary commutative space. In general, for an associative but non-commutative algebra $A$, we can view
$$
HH_\bullet  = \text{Non-commutative differential forms} 
$$
which is a basic algebraic tool for doing calculus in the noncommutative world. 

Let $A=k[y^i]$. Consider the following HKR map
\ali{
\sigma: C_\bullet(A) &\lra  \Omega_A^\bullet \\
f_0\otimes f_1\otimes\cdots\otimes f_m &\lmt f_0 \,df_1\w df_2\w \cdots\w df_m.
}
We can check 
\ali{
&\sigma(b(f_0\otimes f_1\otimes\cdots\otimes f_m)) \\
=&\sigma \LR{f_0 f_1\otimes f_2\otimes\cdots\otimes f_m 
- f_0\otimes f_1f_2\otimes\cdots\otimes f_m \pm\cdots} \\
=& f_0 f_1d f_2\w\cdots\w df_m -f_0 d(f_1f_2)\w df_3\w\cdots\w df_m\pm\cdots \\
 &\quad +(-1)^{m-1}f_0df_1\w\cdots\w d(f_{m-1}f_m) + (-1)^m f_mf_0 df_1\w\cdots\w df_{m-1} \\
=&0
}
so $\sigma$ defines a chain map
$$
\sigma: (C_\bullet(A),b) \lra (\Omega_A^\bullet,0)
$$
from the Hochschild chain complex  to the chain complex of differential forms with zero differential. In fact $\sigma$ is a quasi-isomorphism, which induces an isomorphism by passing to homology
$$
\sigma: HH_\bullet(A) \overset{\simeq}{\lra} \Omega_A^\bullet.
$$
This isomorphism is the one appeared in Theorem \ref{thm-HKR}.

\subsection{Quantum HKR}
Let $\cA=\C[x^1,\cdots,x^r,p_1,\cdots,p_r]$ be the ring of polynomial functions on the phase space $\R^{2r}$. We have a canonical quantization of $\cA$ to the associative Weyl algebra 
$$
\cA^\hbar := (\cA[\hbar],*)
$$
where $*$ is the Moyal product
$$
f*g := f e^{\frac{i}{2}\hbar\sum\limits_{j=1}^r\LR{ {\frac{\overleftarrow\partial}{\partial x^j}} {\frac{\overrightarrow\partial}{\partial p_j}} - {\frac{\overleftarrow\partial}{\partial p_j}} {\frac{\overrightarrow\partial}{\partial x^j}}}} g.
$$
Here ${\frac{\overleftarrow\partial}{\partial(-)}}$ means applying the derivative to the function on the left, and ${\frac{\overrightarrow\partial}{\partial(-)}}$ means applying the derivative to the function on the right. 
In this section, we will describe a quantum version $\sigma^\hbar$ of the HKR map $\sigma$ that intertwines the Hochschild chain complex $C_\bullet(\cA^\hbar)$ of the Weyl algebra $\cA^\hbar $.

\subsubsection{$S^1$-product}
Consider observables constructed in Section \ref{sec:Topological-correlation}
$$
\sO_{f}(t) =  f(\xx+\XX(t),\xp+\PP(t)),\qquad t\in S^1. 
$$
\begin{defn}
Let $f_0,f_1,\cdots,f_m\in\cA$. We define their $S^1$-product 
$$
\bbracket{f_0\otimes f_1\otimes\cdots\otimes f_m}_{S^1}\in 
\cA[\hbar]
$$ 
by the integration of their topological correlation 
\ali{
&\bbracket{f_0\otimes f_1\otimes\cdots\otimes f_m}_{S^1} 
=m! \int_{\Cyc_{m+1}(S^1)} d\theta_0 d\theta_1\cdots d\theta_m  \TopCor{\sO_{f_0}(\theta_0)\sO_{f_1}(\theta_1)\cdots\sO_{f_m}(\theta_m)}.
}
Here the integration is taken over
$$
\Cyc_{m+1}(S^1) = 
\{(\theta_0,\theta_1,\cdots,\theta_m)\in (S^1)^{m+1} \,|\, \theta_0,\theta_1,\cdots,\theta_m \text{ are distinct in clockwise cyclic order}\}.
$$
\end{defn}

\begin{eg}
Given two polynomials $f_0,f_1\in \R[x,p]$,
\ali{
\bbracket{f_0\otimes f_1}_{S^1} 
=& \int_{\Cyc_2(S^1)}d\theta_0 d\theta_1 \, f_0(x,p) e^{i\hbar G(\theta_0,\theta_1) \frac{\overset{\leftarrow}{\partial}}{\partial x}\frac{\overset{\rightarrow}{\partial}}{\partial p} + i\hbar G(\theta_1,\theta_0) \frac{\overset{\leftarrow}{\partial}}{\partial p}\frac{\overset{\rightarrow}{\partial}}{\partial x}} f_1(x,p) \\
=& \int_0^1d\theta \, f_0(x,p) e^{i\hbar (\theta-1/2) \LR{\frac{\overset{\leftarrow}{\partial}}{\partial x}\frac{\overset{\rightarrow}{\partial}}{\partial p} - \frac{\overset{\leftarrow}{\partial}}{\partial p}\frac{\overset{\rightarrow}{\partial}}{\partial x}}} f_1(x,p) \\
=& \int_{-\frac{1}{2}}^{\frac{1}{2}}d\theta \, f_0(x,p) e^{i\hbar \theta  \LR{\frac{\overset{\leftarrow}{\partial}}{\partial x}\frac{\overset{\rightarrow}{\partial}}{\partial p} - \frac{\overset{\leftarrow}{\partial}}{\partial p}\frac{\overset{\rightarrow}{\partial}}{\partial x}}} f_1(x,p) \\
=& \sum_{m=0}^\infty \int_{-\frac{1}{2}}^{\frac{1}{2}}\theta^{2m} d\theta \, \frac{(i\hbar)^{2m}}{(2m)!} f_0(x,p) \LR{\frac{\overset{\leftarrow}{\partial}}{\partial x}\frac{\overset{\rightarrow}{\partial}}{\partial p} - \frac{\overset{\leftarrow}{\partial}}{\partial p}\frac{\overset{\rightarrow}{\partial}}{\partial x}}^{2m} f_1(x,p) \\
=&  \sum_{m=0}^\infty \frac{(i\hbar)^{2m}}{(2m+1)2^{2m}(2m)!} f_0(x,p) \LR{\frac{\overset{\leftarrow}{\partial}}{\partial x}\frac{\overset{\rightarrow}{\partial}}{\partial p} - \frac{\overset{\leftarrow}{\partial}}{\partial p}\frac{\overset{\rightarrow}{\partial}}{\partial x}}^{2m} f_1(x,p).
}
For example,  
\ali{
\bbracket{x\otimes p}_{S^1} &= xp \\
\bbracket{x^2\otimes p^2}_{S^1} &= x^2p^2-\frac{\hbar^2}{6}
}
\end{eg}

\begin{eg}
We use Wick's Theorem to compute
\ali{
\bbracket{x\otimes xp\otimes p}_{S^1} 
=&x^2p^2  
 + 2! \int_{\Cyc_3(S^1)}d\theta_0 d\theta_1 d\theta_2\, i\hbar \LR{G(\theta_0,\theta_1)+G(\theta_1,\theta_2)+G(\theta_0,\theta_2)} xp \\
 &\qquad + 2! \int_{\Cyc_3(S^1)}d\theta_0 d\theta_1 d\theta_2\, (i\hbar)^2  G(\theta_0,\theta_1) G(\theta_1,\theta_2) \\
=& x^2p^2
 + 2xp (i\hbar) \int_{1\ge \theta_1\ge\theta_2\ge 0} d\theta_1 d\theta_2\,  \LR{G(1,\theta_1)+G(\theta_1,\theta_2)+G(1,\theta_2)} \\
 &\qquad +2(i\hbar)^2 \int_{1\ge \theta_1\ge\theta_2\ge 0} d\theta_1 d\theta_2\, G(1,\theta_1) G(\theta_1,\theta_2) \\
=& x^2p^2
 +2 (i\hbar)xp \int_{1\ge \theta_1\ge\theta_2\ge 0} d\theta_1 d\theta_2\, \Lr{\LR{\theta_1-\frac{1}{2}}-\LR{\LR{\theta_1-\theta_2}-\frac{1}{2}}+\LR{\theta_2-\frac{1}{2}}} \\
 &\qquad  + 2(i\hbar)^2 \int_{1\ge \theta_1\ge\theta_2\ge 0} d\theta_1 d\theta_2\, \LR{\theta_1-\frac{1}{2}} \LR{\frac{1}{2}-\LR{\theta_1-\theta_2}} \\
=& x^2p^2+\frac{i\hbar}{6} xp.
}
\end{eg}

\subsubsection{Quantum HKR}

Given $f$, its total differential is 
$$
df= \sum_i \LR{\pp{f}{x^i} dx^i+ \pp{f}{p_i}dp_i}.
$$
We denote 
$$
\sO_{df}(\theta) := \sum_i \LR{\sO_{\partial_{x^i}f}(\theta) dx^i + \sO_{\partial_{p_i}f}(\theta) dp^i}.
$$

\begin{defn}
We define the quantum HKR map
$$
\sigma^\hbar: \cA^{\otimes m+1}\lra \Omega^m_A[\hbar]
$$
by
\ali{
\sigma^\hbar(f_0\otimes f_1\otimes\cdots\otimes f_m) 
:= \int_{\Cyc_{m+1}(S^1)} d\theta_0 d\theta_1\cdots d\theta_m  \TopCor{\sO_{f_0}(\theta_0)\sO_{df_1}(\theta_1)\cdots \sO_{df_m}(\theta_m)}.
}
This can be also formally expressed as the $S^1$-product
$$
\sigma^\hbar(f_0\otimes f_1\otimes\cdots\otimes f_m)
=\frac{1}{m!} \bbracket{f_0\otimes df_1\otimes\cdots\otimes df_m}_{S^1}.
$$
\end{defn}

Using Wick's Theorem (Proposition \ref{prop-Wick-2}), we have the explicit formula 
\ali{
&\sigma^\hbar(f_0\otimes f_1\otimes\cdots\otimes f_m) \\
=&\int_{\Cyc_{m+1}(S^1)} d\theta_0 d\theta_1\cdots d\theta_m \, e^{i\hbar \sum\limits_{i,j=0}^m G(\theta_i,\theta_j) \pp{}{x^{(i)}}\pp{}{p^{(j)}}} f_0(\xx^{(0)},\xp^{(0)})
df_1(\xx^{(1)},\xp^{(1)}) 
\cdots df_m(\xx^{(m)},\xp^{(m)})\bigg|_{\substack{\xx^{(0)}=\cdots=\xx^{(m)}=\xx \\  \xp^{(0)}=\cdots =\xp^{(m)}=\xp}}.
}

\begin{prop}Introduce the following operator 
$$
\Delta: \Omega_A^m \lra \Omega_A^{m-1},\qquad \Delta = \sum_i \pp{}{x^i} \iota_{\partial_{p_i}} - \pp{}{p_i} \iota_{\partial_{x^i}}.
$$
Here $\iota_{V}$ is the contraction with the vector field $V$.  Then $$
\sigma^\hbar(b(-)) = i\hbar\Delta (\sigma^\hbar(-))
$$
where $b$ is the Hochschild differential with respect to the Moyal product. 
\end{prop}
\begin{proof} Let us consider the following integral 
\begin{align*}\label{the-following-integral}
\int_{\Cyc_{m+1}(S^1)} d\theta_0 d\theta_1\cdots d\theta_m   \,& \left\{
 {\pp{}{\theta_1}}\TopCor{\sO_{f_0}(\theta_0)\blue{\sO_{f_1}(\theta_1)}\cdots \sO_{df_m}(\theta_m)}  \nonumber \right.\\
\qquad &- {\pp{}{\theta_2}}\TopCor{\sO_{f_0}(\theta_0)\sO_{df_1}(\theta_1)\blue{\sO_{f_2}(\theta_2)}\cdots \sO_{df_m}(\theta_m)}  \nonumber\\
\qquad &+\cdots \nonumber\\
\qquad & \left.+(-1)^{m-1} {\pp{}{\theta_m}}\TopCor{\sO_{f_0}(\theta_0)\sO_{df_1}(\theta_1)\cdots \sO_{df_{m-1}}(\theta_{m-1})\blue{\sO_{f_m}(\theta_m)}}  \right\}\tag{$\dagger$} 
\end{align*}

We compute (\ref{the-following-integral}) in two ways. The first is to compute $\int d\theta_k \, \pp{}{\theta_k}(-)$ as a total derivative. Due to the cyclic ordering 
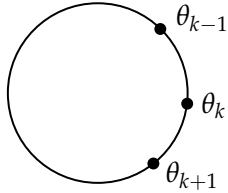
\begin{center}
\tikzset{every picture/.style={line width=0.75pt}} 

\begin{tikzpicture}[x=0.75pt,y=0.75pt,yscale=-1,xscale=1]

\draw   (180.67,120) .. controls (180.67,94.96) and (200.96,74.67) .. (226,74.67) .. controls (251.04,74.67) and (271.33,94.96) .. (271.33,120) .. controls (271.33,145.04) and (251.04,165.33) .. (226,165.33) .. controls (200.96,165.33) and (180.67,145.04) .. (180.67,120) -- cycle ;
\draw  [fill={rgb, 255:red, 0; green, 0; blue, 0 }  ,fill opacity=1 ] (251.56,155.56) .. controls (251.56,154.18) and (252.68,153.06) .. (254.06,153.06) .. controls (255.44,153.06) and (256.56,154.18) .. (256.56,155.56) .. controls (256.56,156.94) and (255.44,158.06) .. (254.06,158.06) .. controls (252.68,158.06) and (251.56,156.94) .. (251.56,155.56) -- cycle ;
\draw  [fill={rgb, 255:red, 0; green, 0; blue, 0 }  ,fill opacity=1 ] (268.62,125.36) .. controls (268.62,123.97) and (269.74,122.86) .. (271.12,122.86) .. controls (272.5,122.86) and (273.62,123.97) .. (273.62,125.36) .. controls (273.62,126.74) and (272.5,127.86) .. (271.12,127.86) .. controls (269.74,127.86) and (268.62,126.74) .. (268.62,125.36) -- cycle ;
\draw  [fill={rgb, 255:red, 0; green, 0; blue, 0 }  ,fill opacity=1 ] (255.06,87.66) .. controls (255.06,86.28) and (256.18,85.16) .. (257.56,85.16) .. controls (258.94,85.16) and (260.06,86.28) .. (260.06,87.66) .. controls (260.06,89.04) and (258.94,90.16) .. (257.56,90.16) .. controls (256.18,90.16) and (255.06,89.04) .. (255.06,87.66) -- cycle ;

\draw (264.05,74) node [anchor=north west][inner sep=0.75pt]   [align=left] {$\theta_{k-1}$};
\draw (276.82,118.15) node [anchor=north west][inner sep=0.75pt]   [align=left] {$\theta_{k}$};
\draw (258.82,153.23) node [anchor=north west][inner sep=0.75pt]   [align=left] {$\theta_{k+1}$};

\end{tikzpicture}
\end{center}
we have 
$$
\int d\theta_k \, \pp{}{\theta_k}(-) =\int_{\theta_{k+1}}^{\theta_{k-1}} d\theta_k \,\pp{}{\theta_k}(-)  = (-) \bigg|_{\theta_{k+1}}^{\theta_{k-1}}.
$$

By Wick's Theorem (Proposition \ref{prop-Wick-2}) and Lemma \ref{lem-limit}, we find the limit behavior
$$
\lim_{\theta_k\to \theta_{k-1}^-} \TopCor{\cdots \sO_f(\theta_{k-1})\sO_g(\theta_k)\cdots} 
=\TopCor{\cdots \sO_{f*g}(\theta_{k-1})\cdots} 
$$
$$
\lim_{\theta_k\to \theta_{k+1}^+} \TopCor{\cdots \sO_f(\theta_{k})\sO_g(\theta_{k+1})\cdots} 
=\TopCor{\cdots \sO_{f*g}(\theta_{k+1})\cdots}
$$
Therefore computing (\ref{the-following-integral}) in terms of total derivative leads to 
\ali{
&\frac{1}{(m-1)!} \bbracket{f_0*f_1\otimes df_2\otimes\cdots\otimes df_m}_{S^1}  -  \frac{1}{(m-1)!} \bbracket{f_0\otimes f_1*df_2\otimes\cdots\otimes df_m}_{S^1} \\
-&\frac{1}{(m-1)!}\bbracket{f_0\otimes df_1* f_2\otimes df_3\otimes\cdots\otimes df_m}_{S^1} + \frac{1}{(m-1)!}\bbracket{f_0\otimes df_1 \otimes f_2*df_3\otimes\cdots\otimes df_m}_{S^1} \\
+&\cdots \\
+&(-1)^{m-1} \frac{1}{(m-1)!} \bbracket{f_0\otimes df_1 \otimes\cdots\otimes df_{m-1}*f_m}_{S^1}  - (-1)^{m-1} \frac{1}{(m-1)!} \bbracket{f_m*f_0\otimes df_1 \otimes\cdots\otimes df_{m-1}}_{S^1}.
}

By the Moyal product formula, we have 
$$
d(f*g) = df*g+f*dg.
$$
Here 
$$
df*g = \sum_i \LR{\partial_{x^i}f*g}dx^i +  \LR{\partial_{p_i}f*g}dp_i 
$$
$$
f*dg = \sum_i \LR{f*\partial_{x^i}g}dx^i +  \LR{f*\partial_{p_i}g}dp_i 
$$

Thus  (\ref{the-following-integral})  becomes 
\ali{
& \frac{1}{(m-1)!}\bbracket{f_0*f_1\otimes df_2\otimes\cdots\otimes df_m}_{S^1}  
 - \frac{1}{(m-1)!} \bbracket{f_0 \otimes d(f_1*f_2)\otimes\cdots\otimes df_m}_{S^1}  
 +\cdots \\
+&(-1)^{m-1} \frac{1}{(m-1)!}\bbracket{f_0\otimes df_1 \otimes\cdots\otimes d(f_{m-1}*f_m)}_{S^1} 
 +(-1)^m \frac{1}{(m-1)!} \bbracket{f_m*f_0\otimes df_1 \otimes\cdots\otimes df_{m-1}}_{S^1} \\
=&\sigma^\hbar(b(f_0\otimes f_1\otimes\cdots\otimes f_m)).
}
Here the Hochschild differential $b$ corresponds to the Moyal product. 

The second way to approach (\ref{the-following-integral}) is to compute $\pp{}{\theta_k}\TopCor{\cdots}$ explicitly. Observe 
$$
\pp{}{\theta_k} G(\theta_k,\theta_j) = -1 = -\pp{}{\theta_k} G(\theta_j,\theta_k) \qquad \text{for }\quad \theta_j\ne \theta_k.
$$
Thus the above explicit formula of $\sigma^\hbar$ implies
\ali{
&\pp{}{\theta_k} \TopCor{\sO_{f_0}(\theta_0)\sO_{df_1}(\theta_1)\cdots \sO_{df_{k-1}}(\theta_{k-1})\sO_{f_k}(\theta_k)\cdots \sO_{f_m}(\theta_m)}  \\
=\;&i\hbar \pp{}{x^i} \TopCor{\sO_{f_0}(\theta_0)\sO_{df_1}(\theta_1)\cdots \sO_{df_{k-1}}(\theta_{k-1})\sO_{\partial_{p_i}f}(\theta_k)\cdots \sO_{df_m}(\theta_m)}  \\
& -i\hbar \pp{}{p_i} \TopCor{\sO_{f_0}(\theta_0)\sO_{df_1}(\theta_1)\cdots \sO_{df_{k-1}}(\theta_{k-1})\sO_{\partial_{x^i}f}(\theta_k)\cdots \sO_{df_m}(\theta_m)}
}
which leads to
$$
\text{(\ref{the-following-integral})}=i\hbar \Delta \sigma^\hbar(f_0\otimes f_1\otimes\cdots\otimes f_m).
$$

Comparing the two computations, we find 
$$
\sigma^\hbar(b(-)) = i\hbar\Delta (\sigma^\hbar(-)).
$$
\end{proof}

We can extend $\sigma^\hbar$ linearly in $\hbar$ to 
$$
\sigma^\hbar: C_\bullet(\cA^\hbar) \lra  \Omega_\cA^\bullet[\hbar] .
$$
Then the above calculation shows that $\sigma^\hbar$ is a morphism of chain complexes
$$
\sigma^\hbar: (C_\bullet(\cA^\hbar),b) \lra (\Omega_\cA^\bullet[\hbar],i\hbar\Delta).
$$
Thus $\sigma^\hbar$ provides a quantization of $\sigma$. It can be further shown that $\sigma^\hbar$ becomes a quasi-isomorphism if we specialize $\hbar$ to any non-zero value, say $\hbar=1$. 
Passing to  $2n$-th homology, $\sigma^\hbar$ leads to the Feigin-Felder-Shoikhet trace formula \cite{FFS}.

\begin{rmk}
Let $\omega\inv=\sum\limits_i \pp{}{x^i}\w \pp{}{p_i}$ be the Poisson bi-vector. Then $\Delta$ is the Lie derivative
$$
\Delta = \LL_{\omega\inv}.
$$
It is clear that $\Delta^2=0$. $\Delta$ is the prototype of  BV (Batalin–Vilkovisky) operator in the quantization of gauge theories \cite{BV}. The above quantum HKR map $\sigma^\hbar$ gives an equivalent description of the BV quantization of topological quantum mechanics \cite{GLX}, which leads to a QFT  approach to the algebraic index theorem as developed by Fedosov \cite{Fedosov-book} and Nest-Tsygan \cite{Nest-Tsygan}. 
\end{rmk}

\end{document}